\documentclass[twoside,leqno,10pt, A4]{amsart}
\usepackage{amsfonts}
\usepackage{amsmath}
\usepackage{amscd}
\usepackage{amssymb}
\usepackage{amsthm}
\usepackage{amsrefs}
\usepackage{latexsym}
\usepackage{mathrsfs}
\usepackage{bbm}
\usepackage{amscd}
\usepackage{amssymb}
\usepackage{amsthm}
\usepackage{amsrefs}
\usepackage{latexsym}
\usepackage{mathrsfs}
\usepackage{bbm}
\usepackage{enumerate}
\usepackage{graphicx}
\usepackage{color}
\begin{document}

\newtheorem{theorem}[subsection]{Theorem}
\newtheorem{proposition}[subsection]{Proposition}
\newtheorem{lemma}[subsection]{Lemma}
\newtheorem{corollary}[subsection]{Corollary}
\newtheorem{conjecture}[subsection]{Conjecture}
\newtheorem{prop}[subsection]{Proposition}
\numberwithin{equation}{section}
\newcommand{\mr}{\ensuremath{\mathbb R}}
\newcommand{\mc}{\ensuremath{\mathbb C}}
\newcommand{\dif}{\mathrm{d}}
\newcommand{\intz}{\mathbb{Z}}
\newcommand{\ratq}{\mathbb{Q}}
\newcommand{\natn}{\mathbb{N}}
\newcommand{\comc}{\mathbb{C}}
\newcommand{\rear}{\mathbb{R}}
\newcommand{\prip}{\mathbb{P}}
\newcommand{\uph}{\mathbb{H}}
\newcommand{\fief}{\mathbb{F}}
\newcommand{\majorarc}{\mathfrak{M}}
\newcommand{\minorarc}{\mathfrak{m}}
\newcommand{\sings}{\mathfrak{S}}
\newcommand{\fA}{\ensuremath{\mathfrak A}}
\newcommand{\mn}{\ensuremath{\mathbb N}}
\newcommand{\mq}{\ensuremath{\mathbb Q}}
\newcommand{\half}{\tfrac{1}{2}}
\newcommand{\f}{f\times \chi}
\newcommand{\summ}{\mathop{{\sum}^{\star}}}
\newcommand{\chiq}{\chi \bmod q}
\newcommand{\chidb}{\chi \bmod db}
\newcommand{\chid}{\chi \bmod d}
\newcommand{\sym}{\text{sym}^2}
\newcommand{\hhalf}{\tfrac{1}{2}}
\newcommand{\sumstar}{\sideset{}{^*}\sum}
\newcommand{\sumprime}{\sideset{}{'}\sum}
\newcommand{\sumprimeprime}{\sideset{}{''}\sum}
\newcommand{\sumflat}{\sideset{}{^\flat}\sum}
\newcommand{\shortmod}{\ensuremath{\negthickspace \negthickspace \negthickspace \pmod}}
\newcommand{\V}{V\left(\frac{nm}{q^2}\right)}
\newcommand{\sumi}{\mathop{{\sum}^{\dagger}}}
\newcommand{\mz}{\ensuremath{\mathbb Z}}
\newcommand{\leg}[2]{\left(\frac{#1}{#2}\right)}
\newcommand{\muK}{\mu_{\omega}}
\newcommand{\thalf}{\tfrac12}
\newcommand{\lp}{\left(}
\newcommand{\rp}{\right)}
\newcommand{\Lam}{\Lambda_{[i]}}
\newcommand{\lam}{\lambda}
\newcommand{\af}{\mathfrak{a}}
\newcommand{\fb}{\mathfrak{b}}
\newcommand{\sw}{S(X,Y;\phi,\psi)}
\newcommand{\lz}{\left(}
\newcommand{\pz}{\right)}
\newcommand{\bfrac}[2]{\lz\frac{#1}{#2}\pz}
\newcommand{\odd}{\mathrm{\ primary}}
\newcommand{\even}{\text{ even}}
\newcommand{\res}{\mathrm{Res}}

\theoremstyle{plain}
\newtheorem{conj}{Conjecture}
\newtheorem{remark}[subsection]{Remark}

\makeatletter
\def\widebreve{\mathpalette\wide@breve}
\def\wide@breve#1#2{\sbox\z@{$#1#2$}%
     \mathop{\vbox{\m@th\ialign{##\crcr
\kern0.08em\brevefill#1{0.8\wd\z@}\crcr\noalign{\nointerlineskip}%
                    $\hss#1#2\hss$\crcr}}}\limits}
\def\brevefill#1#2{$\m@th\sbox\tw@{$#1($}%
  \hss\resizebox{#2}{\wd\tw@}{\rotatebox[origin=c]{90}{\upshape(}}\hss$}
\makeatletter

\title[One-level density of quadratic twists of $L$-functions]{One-level density of quadratic twists of $L$-functions}

\author{Peng Gao and Liangyi Zhao}

\begin{abstract}
 In this paper, we investigate the one-level density of low-lying zeros of quadratic twists of automorphic $L$-functions under the generalized Riemann hypothesis and the Ramanujan-Petersson conjecture. We improve upon the known results using only functional equations for quadratic Dirichlet $L$-functions.
\end{abstract}

\maketitle



\section{Introduction}
\label{sec 1}

  According to the Langlands program (see \cite{Langlands}), the most general $L$-function is that attached to an
automorphic representation of $\text{GL}_N$ over a number field, which in turn can be written as products of the $L$-functions attached to cuspidal automorphic representations of $\text{GL}_M$ over $\mq$. \newline

  Let $L(s, \pi)$ be the $L$-function attached to an automorphic cuspidal representation $\pi$ of $\text{GL}_{M}$ over $\mq$. For $\Re(s)$
large enough, $L(s, \pi)$ can be expressed as an Euler product of the form
\begin{align*}
 L(s, \pi)=\prod_pL(s, \pi_p)=\prod_p\prod^M_{j=1}(1-\alpha_{\pi}(p, j)p^{-s})^{-1}.
\end{align*}

The order of vanishing of $L(s, \pi)$ at the central point has rich arithmetic applications as can be seen from the Birch and Swinnerton-Dyer
conjecture on the rank of an elliptic curve. One way to address the non-vanishing issue is to apply the density conjecture of N. Katz
and P. Sarnak \cites{KS1, K&S}, which suggests that the distribution of zeros near the central point of a family of $L$-functions is the same as that of eigenvalues near $1$ of a corresponding classical compact group.  The truth of the density conjecture would imply that $L(\half, \chi) \neq 0$ for almost all quadratic Dirichlet characters $\chi$. \newline

In this direction, it was estbalished unconditionally by K. Soundararajan \cite{sound1} that $L(\half, \chi_{8d}) \neq 0$ for at least 87.5\% of odd square-free $d \geq 0$.
Here and after, $\chi_d=\leg {d}{\cdot}$ stands for the Kronecker symbol.  Conditionally under the truth of the generalized Riemann hypothesis (GRH), A. E. \"{O}zluk and C. Snyder \cite{O&S} computed the one level density for the family of quadratic Dirichlet $L$-functions. Their result implies that, after optimizing the test function involved as done in \cites{B&F, ILS}, we have $L(1/2, \chi_d) \neq 0$ for at least $(19-\cot \frac{1}{4})/16 = 94.27\ldots \%$ of the fundamental discriminants $|d| \leq X$. \newline

  The work of  A. E. \"{O}zluk and C. Snyder has been extended by M. O. Rubinstein \cite{Ru} to all $n$-level densities for families of
quadratic twists of automorphic $L$-functions. In \cite{Gao}, the first-named author computed the $n$-level densities for the family of quadratic Dirichlet $L$-functions under GRH. In \cite{G&Zhao2}, the authors studied the one-level
density for the family for quadratic twists of modular $L$-functions under GRH. In \cite{ER-G&R}, A. Entin, E. Roditty-Gershon and Z. Rudnick
considered the $n$-level densities for the family for quadratic Dirichlet $L$-functions over function field. For other works on the $n$-level
densities of families of quadratic twists of $L$-functions and their relations to the random matrix theory, see \cites{A&B, BFK21-1, B&F, MS,
MS1, Miller1, FPS,FPS1,FPS2, HMM,HKS, G&Zhao4, G&Zhao5, G&Zhao6, G&Zhao9}. \newline

   The aim of this paper to compute the one-level density for the family for quadratic twists of automorphic $L$-functions under GRH.
For this, we fix a self-contragredient representation $\pi$ of $\text{GL}_{M}$ over $\mq$ so that $\pi = \tilde{\pi}$. As the cases of $M=1,2$ have been studied in \cite{O&S, G&Zhao2}, we assume that $M \geq 3$. We also assume the Ramanujan-Petersson conjecture, which implies that
\begin{equation}
\label{Ramanujanconj}
   |\alpha_{\pi}(p, j)| \leq 1.
\end{equation}

    The Rankin-Selberg symmetric square $L$-function $L(s, \pi \otimes \pi)$ factors as the product of the symmetric and exterior square $L$-functions
( see \cite[p. 139]{BG}):
\begin{align}
\label{RankinSelbergdecomp}
 L(s, \pi \otimes \pi)=L(s, \vee^2)L(s, \wedge^2),
\end{align}
   and has a simple pole at $s = 1$  which is carried by one of the two factors. Write the
order of the pole of $L(s, \wedge^2)$ as $(\delta(\pi) + 1)/2$  so that $\delta(\pi) = \pm 1$. The order of the pole of $L(s, \vee^2)$ then equals to $(1-\delta(\pi) )/2$.   In this paper, we restrict our attention to those $\pi$ with $\delta(\pi)=-1$.  The case in which $\delta(\pi)=1$ will be discussed later at the end of this section. \newline

   We are interested in the following family of quadratic twists of $L$-functions given by
\begin{align*}
 \mathcal F = \big\{ L(s, \pi \otimes \chi_{8d}) : d \text{ positive, odd and square-free} \big\}.
\end{align*}

   Note that $\chi_{8d}$ is a primitive quadratic Dirichlet character modulo $8d$ when $d$ is positive, odd and square-free. Any $L(s, \pi
\otimes \chi_{8d}) \in \mathcal F$ has an Euler product given by (see \cite[Section 3.6]{Ru})
\begin{align*}
 L(s, \pi \otimes \chi_{8d})=\prod_p\prod^M_{j=1}(1-\chi_{8d}(p)\alpha_{\pi}(p, j)p^{-s})^{-1}.
\end{align*}

  As $\pi = \tilde{\pi}$, $L(s, \pi \otimes \chi_{8d})$  has a functional equation of the form
\begin{align*}
 \Lambda(s,  \pi \otimes \chi_{8d}):=&\pi^{-Ms/2}\prod^M_{j=1}\Gamma \Big( \frac {s+\mu_{\pi \otimes \chi_{8d}}(j)}{2}\Big) L(s, \pi \otimes
\chi_{8d}) = \epsilon(s, \pi \otimes \chi_{8d})\Lambda(1-s,  \pi \otimes \chi_{8d}),
\end{align*}
  where $\mu_{\pi \otimes \chi_{8d}}(j) \in \mc$ and satisfies $\Re(\mu_{\pi \otimes \chi_{8d}}(j)) \geq 0$ since we are assuming GRH. \newline

   Moreover, as $\delta(\pi) =-1$, we have
\begin{align*}
 \epsilon(s, \pi \otimes \chi_{8d})=Q^{-s+1/2}_{\pi \otimes \chi_{8d}},
\end{align*}
  where $Q_{\pi \otimes \chi_{8d}}$ is the conductor of $L(s, \pi \otimes \chi_{8d})$ and it is known (see \cite[Section 2]{M&T-B}) that
\begin{align}
\label{Qbound}
  Q_{\pi \otimes \chi_{8d}} \ll d^M.
\end{align}

  As we assume GRH, we may write the zeros of $\Lambda(s,  \pi \otimes \chi_{8d})$ as $\half+i \gamma_{\pi \otimes \chi_{8d}, j}$ with
$\gamma_{\pi \otimes \chi_{8d}, j} \in \mr$ and we order them as
\begin{equation*}
    \ldots \leq
    \gamma_{\pi \otimes \chi_{8d}, -2} \leq
   \gamma_{\pi \otimes \chi_{8d}, -1} < 0 \leq  \gamma_{\pi \otimes \chi_{8d}, 1} \leq  \gamma_{\pi \otimes \chi_{8d}, 2} \leq
   \ldots.
\end{equation*}

   Let $X$ be a large real number and we normalize the zeros by defining
\begin{align*}
    \tilde{\gamma}_{\pi \otimes \chi_{8d}, j}= \frac{M\gamma_{\pi \otimes \chi_{8d}, j}}{2 \pi} \log X.
\end{align*}
 We define for a fixed even Schwartz class function $\phi$,
\begin{equation*}
S(\chi_{8d}; \pi, X, \phi)=\sum_{j} \phi(\tilde{\gamma}_{\pi \otimes \chi_{8d}, j}).
\end{equation*}

  We shall regard the function $\phi$ in the above expression as a test function and we also fix a non-zero, non-negative function $w$
compactly supported on $\mr^+$, which we shall regard as a weight function.  We then define the one-level density of the family $\mathcal F $
with respect to $w$ by
\begin{align}
\label{D}
  D(X;\phi, w, \mathcal F) :=\frac 1{W(X)}\sumstar_{(d,2)=1}  w\left( \frac {d}X \right)  S(\chi_{8d}; \pi, X, \phi),
\end{align}
    where we use $\sum^*$ to denote a sum over square-free integers throughout the paper and $W(X)$ is the total weight given by
\begin{align*}
  W(X)=\sumstar_{(d,2)=1}  w\left( \frac {d}X \right).
\end{align*}

   We evaluate the one-level density $D(X;\phi, w, \mathcal F)$ asymptotically in this paper and our result is as follows.
\begin{theorem}
\label{quadraticmainthm}
Suppose that GRH and the Ramanujan-Petersson conjecture are true. Let $M \geq 2$ and let $\pi$ be a fixed self-contragredient representation
of $\text{GL}_{M}$ over $\mq$ such that $\delta(\pi)=-1$.  Let $w$ be a non-negative function that is not identically zero and compactly supported on $\mr^+$
and let $\phi(x)$ be an even
Schwartz function whose Fourier transform $\hat{\phi}(u)$ has
compact support in $(-2/M, 2/M)$.  Then we have
\begin{align}
\label{theo1eq}
 \lim_{X \rightarrow +\infty}D(X;\phi, w, \mathcal F)
 = \int\limits_{\rear} \phi(x)  W_{USp}(x) \dif x, \quad \mbox{where} \quad  W_{USp}(x) =1-\frac{\sin(2\pi x)}{2\pi x}. 
\end{align}
\end{theorem}

  Note that the kernel $W_{USp}$ in the integral of \eqref{theo1eq} is the same one appearing on the random matrix theory side, when studying the eigenvalues of unitary symplectic matrices. This implies that the family $\mathcal F$ of quadratic twists of $L$-functions is a symplectic family and hence verifies the density conjecture of N. Katz and P. Sarnak for this family when the support of the Fourier transform of the test function $\phi$ is contained in the interval $(-2/M, 2/M)$. Our result improves upon a result of Rubinstein given in \cite[Theorem 3.2]{Ru} by doubling the size of the allowable support of $\hat \phi$ in the case of the one-level density. \newline

  Theorem \ref{quadraticmainthm} can be regarded as a generalization of the result of \"{O}zluk and Snyder \cite{O&S} on the one-level density of the family of quadratic Dirichlet $L$-functions and Theorem 1.1 of \cite{G&Zhao2} on the one-level density of the family of quadratic
twists of modular $L$-functions. A key step in the proofs of all cases is to estimate certain character sums. The treatments in \cite{O&S} and \cite{G&Zhao2} use Poisson summations to convert the character sums into corresponding dual sums (see the work of K. Soundararajan \cite{sound1} for a systematic development of this Poisson summation formula). Our proof of Theorem \ref{quadraticmainthm} is different as we only makes use of functional equations of the $L$-functions involved, although we shall call this approach Poisson summation as well. In fact, our approach here is motivated by a Poisson summation formula over number fields established by L. Goldmakher and B. Louvel in \cite[Lemma 3.2]{G&L}. The equivalence of this version of Poisson summation formula and that developed by Soundararajan \cite{sound1} over primitive Hecke character of the Gaussian filed $\mq(i)$ has already been pointed out in \cite[Section 2.6]{Gao2}. More generally, one may evaluate the expression for $D(X;\phi, w, \mathcal F)$ in \eqref{D} by first applying the explicit formula given in Lemma \ref{lem2.4} to convert the sum over zeros of $L$-function into a sum over prime powers. The resulting double sum can then be viewed as a double Dirichlet series and the use of functional equations can be thought of as an effort to understand the analytical behavior of the double Dirichlet series. One may find in \cite{DGH} a build-up of the theory of double Dirichlet series. \newline

We remark here that Theorem~\ref{quadraticmainthm} may be extended to the case in which $\delta(\pi)=1$.  This requires generalizing the Possion summation formula in Lemma \ref{Poissonsum} to smoothed character sums over arithmetic progressions.  As noted in \cite[pp. 173--176]{Ru}, the sign of the functional equation of $L(s,\pi \otimes \chi_d)$ depends on $d$ lying in certain fixed arithmetic progressions.

\section{Preliminaries}
\label{sec 2}

\subsection{The Explicit Formula}
\label{section: Explicit Formula}

Our approach in this paper relies on the following explicit
formula given in \cite[(3.65)]{Ru}, which essentially converts a sum over zeros of an
$L$-function to a sum over primes.
\begin{lemma}
\label{lem2.4}
   Let $\phi(x)$ be an even Schwartz function whose Fourier transform
   $\hat{\phi}(u)$ has compact support. For any odd, positive and square-free integer $d$, we have
\begin{align}
\label{Schiphi}
S(\chi_{8d}; \pi, X, \phi) =\int\limits^{\infty}_{-\infty}  \phi(t) \dif t-\frac {2}{M\log X}\sum^{\infty}_{m=1} \left(\frac
{\Lambda(m)a_{\pi}(m)}{\sqrt{m}} \chi_{8d}(m) \hat{\phi}\left( \frac { \log m}{M\log X} \right) \right)+O\left(\frac{1}{\log
X}\right),
\end{align}
   where
\begin{equation*}
   a_{\pi}(p^k)=\sum^M_{j=1}\alpha^k_{\pi}(p, j).
\end{equation*}
\end{lemma}

   Note that the Ramanujan-Petersson conjecture \eqref{Ramanujanconj} implies that $|a_{\pi}(p^k)| \leq M$ for all $p$ and $k$. It follows that the terms corresponding to $m=p^k$ with $k \geq 3$ in \eqref{Schiphi} contributes $O(1/(\log X))$.
   Moreover, note that by \cite[Lemma 3.9]{Levinson},
\begin{equation*}
  \sum_{p | 8d}\frac {\log p}{p} \ll \log \log (3d).
\end{equation*}

    We then deduce that the terms corresponding to $m=p^2$ in \eqref{Schiphi}  can be written as
\begin{equation*}
-\frac {2}{M\log X}\sum_{p} \frac {a_{\pi}(p^2)\log p }{p}  \hat{\phi}\left( \frac { 2\log p}{M\log X} \right) +
O \left( \frac {\log \log 3d}{\log  X} \right).
\end{equation*}

    To estimate the above sum, we note the Euler products for $L(s, \pi \otimes \pi)$ and $L(s, \vee^2)$  can be found in \cite[(2.18)]{R&S} and on \cite[p. 167]{BG}, respectively. Thus, we see by \eqref{RankinSelbergdecomp} that if we write
\begin{align}
\label{LEulerprod}
  L(s, \vee^2) = \prod_p L_p(s, \vee^2), \quad \mbox{and} \quad L(s, \wedge^2)= \prod_p L_p(s, \wedge^2),
\end{align}
then we have for unramified $p$,
\begin{align*}
  L_p(s, \vee^2) =& \prod_{1 \leq i \leq j \leq M}(1-\alpha_{\pi}(p, i)\alpha_{\pi}(p, j)p^{-s})^{-1}, \quad \mbox{and}  \\
  L_p(s, \wedge^2)=& \prod_{1 \leq i < j \leq M}(1-\alpha_{\pi}(p, i)\alpha_{\pi}(p, j)p^{-s})^{-1}.
\end{align*}

   Upon taking the logarithmic derivatives on both sides of the expressions in \eqref{LEulerprod} with the above in mind, we deduce from  \cite[Theorem 5.15]{iwakow} that under GRH,
\begin{align*}
  \sum_{p \leq x} \sum_{1 \leq i \leq j \leq M}\alpha_{\pi}(p, i)\alpha_{\pi}(p, j)\log p=& \frac {1-\delta(\pi)}{2}x+O(x^{1/2}\log^{2} x), \quad \mbox{and} \\
  \sum_{p \leq x} \sum_{1 \leq i < j \leq M}\alpha_{\pi}(p, i)\alpha_{\pi}(p, j)\log p=& \frac {\delta(\pi) + 1}{2}x+O(x^{1/2}\log^{2} x),
\end{align*}
  where the implied constants depend on $\pi$ only. \newline

  Taking difference of the above expressions then implies that under GRH,
\begin{align}
\label{sumpest}
   \sum_{p \leq x}a_{\pi}(p^2)\log p =\sum_{p \leq x}\sum_{1  \leq j \leq M}\alpha^2_{\pi}(p, j)\log p =& -\delta(\pi)x+O(x^{1/2}\log^{2} x).
\end{align}

   We now apply \eqref{sumpest} (by noting that $\delta(\pi)=-1$) and partial summation to get
\begin{align*}
 -\frac {2}{M\log X} & \sum_{p} \frac {a_{\pi}(p^2)\log p }{p}  \hat{\phi}\left( \frac { 2\log p}{M\log X} \right) \\
=& -\frac {2}{M\log X}
\int\limits^{\infty}_{1}\hat{\phi}\left( \frac { 2\log t}{M\log X} \right) \frac {\dif t}{t}-\frac {2}{M\log X}
\int\limits^{\infty}_{1}\hat{\phi}\left( \frac { 2\log t}{M\log X} \right) \frac {\dif O(t^{1/2}\log^{2} t)}{t} \\
=&-\frac 12\int\limits^{\infty}_{-\infty}\hat{\phi}(t) \dif t+O \Big( \frac 1{\log X} \Big).
\end{align*}

  Our discussions above allow us to simplify the expression for $S(\chi_{8d}; \pi, X, \phi)$ in \eqref{Schiphi} in the following way.
\begin{lemma}
\label{lem2.1}
   Let $\phi(x)$ be an even Schwartz function whose Fourier transform $\hat{\phi}(u)$ has compact support. For any odd, positive and square-free integer $d$, we have
\begin{align}
\label{Schiphisimplified}
  S(\chi_{8d}; \pi, X, \phi) = \int\limits^{\infty}_{-\infty}\phi(t) \dif t-\frac 1{2}
   \int\limits^{\infty}_{-\infty}\hat{\phi}(u) \dif u-S_1(\chi_{8d}; \pi, X, \hat{\phi})+O \left( \frac {\log \log 3X}{\log
   X} \right),
\end{align}
   where the implicit constant in the $O$-term depends on $\phi$ alone and
\begin{align*}
S_1(\chi_{8d}; \pi, X, \phi) =\frac {2}{M\log X}\sum_{p} \frac { a_{\pi}(p)\log p}{\sqrt{p}} \chi_{8d}(p) \hat{\phi}\left( \frac { \log
p}{M\log X} \right) .
\end{align*}
\end{lemma}

\subsection{Poisson summation}
\label{sec Poisson}

   In this section we develop a Poisson summation formula for a primitive quadratic Dirichlet character $\chi=\leg {\cdot}{q}$, where
$q$ is an odd prime. This is already contained in \cite[Lemma 3.2]{G&L} and we include the details here for the completeness.  We let $\af=0$ or $1$ be given by $\chi(-1)=(-1)^{\af}$ and define
\begin{align}
\label{Lambdasdef}
  \Lambda(s, \chi) = \Big( \frac {q}{\pi} \Big)^{s/2} \Gamma \Big( \frac {s+\af}{2} \Big)L(s, \chi).
\end{align}

    Then we have the following functional equation (see \cite[\S 9, (14)]{Da})
\begin{align}
\label{fcneqn}
  \Lambda(s, \chi) = \frac {\tau(\chi)}{i^{\af}\sqrt{q}}\Lambda(1-s, \overline{\chi}),
\end{align}
   where $\tau(\chi)$ is the Gauss sum associated to $\chi$ defined by
\begin{equation*}
  \tau(\chi) =  \sum_{1 \leq x \leq q}\chi(x) e \left( \frac{x}{q} \right), \quad \mbox{where} \quad e(z) = \exp (2 \pi i z).
\end{equation*}
   Note that $\overline{\chi}=\chi$ and that $\tau(\chi)=i^{\af}\sqrt{q}$ (see \cite[\S 2]{Da}), so that the functional equation given
in \eqref{fcneqn} becomes
\begin{align}
\label{fcneqnquad}
  \Lambda(s, \chi) = \Lambda(1-s, \chi).
\end{align}

  We recall that the Mellin transform $\mathcal{M}f$ for any function $f$ is defined to be
\begin{align*}
     \mathcal{M}f(s) =\int\limits^{\infty}_0f(t)t^s\frac {\dif t}{t}.
\end{align*}

   Now, for any smooth function $W:\mr^{+} \rightarrow \mr$ of compact support, we further define for $c>0$, $\fb=0$ or $1$,
\begin{align}
\label{WMell}
\begin{split}
 &\widetilde{W}_{\fb}(x)=  \frac 1{2\pi i}\int\limits\limits_{(c)}\mathcal{M}W(1-u)x^{-u}\pi^{-(2u-1)/2}\Gamma \Big( \frac{u+\fb}{2} \Big) \Gamma^{-1}\Big( \frac {1-u+\fb}{2} \Big) \dif u.
\end{split}
\end{align}
It goes without saying that $\pi$ appearing in this section denotes the number $3.14\ldots$ and not the self-contragredient representation under our consideration. \newline

   We note that integration by parts implies that for any integer $E>0$ and any $s \in \mc$,
\begin{align*}
\begin{split}
\mathcal{M}W(s) \ll (1+|s|)^{-E}.
\end{split}
\end{align*}
   It follows from this and Stirling's formula that the integration in \eqref{WMell} is convergent. \newline

  Moreover, by pairing together the $u$ and $\overline{u}$ values of the integrand in \eqref{WMell}, we see that $\widetilde{W}_{\fb}(x)$ is
real-valued for $x \in \rear^+$. Upon shifting the contour of integration in  \eqref{WMell} to the left and right according to whether $x \geq 1$ or
not, and observing that $\Gamma(s)$ has a pole at $s=0$, we see that for any integer $j \geq 0$ and any real $A>0$,
\begin{align*}
\begin{split}
 &\widetilde{W}^{(j)}_{\fb}(x) \ll_{j,A} \min (1, x^{-A}).
\end{split}
\end{align*}

   Our next result gives a Poisson summation formula for smoothed character sums over integers in $\mz$.
\begin{lemma}
\label{Poissonsum} Let $q$ be an odd prime and let $\chi=\leg {\cdot}{q}$. For any smooth function $W:\mr^{+} \rightarrow \mr$ of compact
support,  we have for $X>0$,
\begin{align}
\label{PoissonsumQw}
   \sum_{n}\chi(n)W\left(\frac {n}{X}\right) = \frac {X}{\sqrt{q}} \sum^{\infty}_{m=1}\chi(m)\widetilde{W}_{\af} \Big( \frac {mX}{q} \Big).
\end{align}
\end{lemma}
\begin{proof}
  We apply the inverse Mellin transform to write $W(t)$ as, for some $c_u>1$,
\begin{align*}
    W \left(t \right)=\frac 1{2\pi i}\int\limits\limits_{(c_u)}\mathcal{M}W(u)t^{-u} \dif u.
\end{align*}

    It follows that
\begin{align}
\label{chisum}
   \sum_{n}\chi(n)W\left(\frac {n}{X}\right)=\frac 1{2\pi i}\int\limits\limits_{(c_u)}\mathcal{M}W(u)X^{u}\Big (\sum_{n}\frac {\chi(n)}{n^{u}}
\Big ) \dif u =\frac 1{2\pi i}\int\limits\limits_{(c_u)}\mathcal{M}W(u)X^{u}L(u, \chi) \dif u.
\end{align}

    Applying the definition of $\Lambda(s, \chi)$ in \eqref{Lambdasdef} to \eqref{chisum} yields
\begin{align*}
  \sum_{n}\chi(n)W\left(\frac {n}{X}\right) =\frac 1{2\pi i}\int\limits\limits_{(c_u)}\widehat{W}(u)X^{u} \Big( \frac{q}{\pi} \Big)^{-u/2}\Gamma^{-1} \Big( \frac {u+\af}{2} \Big) \Lambda(u, \chi) \dif u.
\end{align*}

     We move the line of integration above to $\Re(u)=c'_u$ for some $c'_u<0$, encountering no poles in the process. Thus, we
obtain that
\begin{align*}
  \sum_{n}\chi(n)W\left(\frac {n}{X}\right) =\frac 1{2\pi i}\int\limits\limits_{(c'_u)}\mathcal{M}W(u)X^{u} \Big( \frac{q}{\pi} \Big)^{-u/2} \Gamma^{-1}\Big(\frac {u+\af}{2} \Big) \Lambda(u, \chi) \dif u.
\end{align*}

   We apply the functional equation given in \eqref{fcneqnquad} to recast the expression above as
\begin{align*}
\begin{split}
  \sum_{n}\chi(n)W\left(\frac {n}{X}\right) =&  \frac 1{2\pi i}\int\limits\limits_{(c'_u)}\mathcal{M}W(u)X^{u} \Big( \frac
{q}{\pi} \Big)^{-u/2}\Gamma^{-1} \Big( \frac {u+\af}{2} \Big) \Lambda(1-u, \chi) \dif u \\
=& \frac 1{2\pi i}\int\limits\limits_{(c'_u)}\mathcal{M}W(u)X^{u} \Big( \frac {q}{\pi} \Big)^{(1-2u)/2}\Gamma^{-1} \Big( \frac {u+\af}{2} \Big) \Gamma \Big( \frac{1-u+\af}{2} \Big) L(1-u, \chi) \dif u.
\end{split}
\end{align*}

  We make a change of variable $u \rightarrow 1-u$ in the last expression above to see that for some $c>1$,
\begin{align*}
\begin{split}
 & \sum_{n}\chi(n)W\left(\frac {n}{X}\right) = \frac 1{2\pi i}\int\limits\limits_{(c)}\mathcal{M}W(1-u)X^{1-u} \Big( \frac
{q}{\pi} \Big)^{(2u-1)/2} \Gamma \Big( \frac {u+\af}{2} \Big) \Gamma^{-1} \Big( \frac {1-u+\af}{2} \Big)
L(u, \chi) \dif u.
\end{split}
\end{align*}

Now substituting in the Dirichlet series for the $L$-function, interchanging the sum and integral leads to \eqref{PoissonsumQw}. This completes the proof.
\end{proof}

\section{Proof of Theorem \ref{quadraticmainthm}}

\subsection{Setup}
\label{sec 3.0}

   Applying the explicit formula \eqref{Schiphisimplified} in \eqref{D} and using the M\"obius function $\mu$ to detect square-free $d$'s, we get
\begin{align}
\label{Dlimit}
    \lim_{X \rightarrow +\infty}D(X;\phi, w, \mathcal F)=\int\limits^{\infty}_{-\infty}\phi(t) \dif t-\frac 1{2}
   \int\limits^{\infty}_{-\infty}\hat{\phi}(u) \dif u-\frac {2}{M}\lim_{X \rightarrow \infty}\frac {  S(X; \pi, \hat{\phi},  W) }{W(X)\log X },
\end{align}
  where
\begin{align}
\label{SX}
    S(X; \pi, \hat{\phi},  W) :=
    \sum_{(d,2)=1}\mu^2(d)\sum_{p}\frac {a_{\pi}(p)\log p}{\sqrt {p}} \left( \frac{8d}{p} \right) \hat{\phi} \left( \frac {\log p}{M\log X}
\right) W \left( \frac {d}{X} \right).
\end{align}

   Note that we have by \cite[(2.2)]{FPS} that as $X \rightarrow \infty$,
 \[ W(X) \sim \frac{4X}{\pi^2}\widehat{W}(1). \]

   Observe also that, $M \geq 2$ implies that the support of $\hat{\phi}$ is contained in $(-2/M, 2/M) \subset (-1,1)$.  Thus we have
\begin{align}
\label{phiexp}
  \int\limits^{\infty}_{-\infty}\phi(t) \dif t-\frac 1{2}
   \int\limits^{\infty}_{-\infty}\hat{\phi}(u) \dif u=\int\limits^{\infty}_{-\infty}\phi(t)\left( 1-\frac {\sin (2 \pi t)}{2 \pi
   t} \right) \dif t.
\end{align}

   We then deduce from \eqref{Dlimit} and \eqref{phiexp} that in order to prove Theorem \ref{quadraticmainthm}, it suffices to show that
\begin{align}
\label{Sorder}
    S(X; \pi, \hat{\phi},  W) =o(X\log X),
\end{align}
as $X \to \infty$. \newline

     To this end, we may assume that the support of  $\hat{\phi}$  is contained in the interval $(-2/M+\varepsilon/M, 2/M-\varepsilon/M)$ for any $\varepsilon>0$. This then implies that we may also restrict the sum over $p$ in \eqref{SX} to $p \leq Y$, where we set $Y=X^{2-\varepsilon}$. Let $Z >1$ be a real parameter to be chosen later and we write $\mu^2(d)=M_Z(d)+R_Z(d)$ where
\begin{equation*}
    M_Z(d)=\sum_{\substack {l^2|d \\ l \leq Z}}\mu(l) \quad \mbox{and} \quad R_Z(d)=\sum_{\substack {l^2|d \\ l >  Z}}\mu(l).
\end{equation*}

Now set
\[ S_M(X; \pi, \hat{\phi}, W) = \sum_{(d,2)=1}M_Z(d)\sum_{p}
    \frac {a_{\pi}(p)\log p}{\sqrt {p}} \left( \frac{8d}{p} \right) \hat{\phi} \left( \frac {\log p}{M\log X} \right) W\left( \frac{d}{X}
\right) \]
    and
\[ S_R(X; \pi, \hat{\phi}, W)
=\sum_{(d,2)=1}R_Z(d)\sum_{p}
    \frac {a_{\pi}(p)\log p}{\sqrt {p}}\left( \frac{8d}{p} \right) \hat{\phi} \left( \frac {\log p}{M\log X}\right) W\left( \frac {d}{X}
\right). \]
Thus $S(X; \pi, \hat{\phi},  W) =S_M(X; \pi, \hat{\phi},W)+S_R(X; \pi, \hat{\phi},W)$.

\subsection{Estimation of $S_R(X; \pi, \hat{\phi},W)$}
\label{sec 3.1}
For any odd, positive integer $q$, we write $q=q_1q^2_2$ with $q_1$ square-free and apply the bound in \eqref{Ramanujanconj} to obtain that for any real $x \geq 1$,
\begin{equation*}
  \sum_{p \leq x} a_{\pi}(p)\chi_{8q} (p)\log p-\sum_{p \leq x} a_{\pi}(p)\chi_{8q_1} (p)\log p \ll \sum_{p | q} |a_{\pi}(p)\chi_{8q}(q)| \log p \ll \log q.
\end{equation*}

 We further apply \cite[Theorem 5.15]{iwakow} and \eqref{Qbound} to bound the sum
 \[ \sum_{p \leq x} a_{\pi}(p)\chi_{8q_1} (p)\log p \]
 under GRH.   The main term in \cite[(5.56)]{iwakow} equals to $0$ in our case since we have $M \geq 2$ here.  We thus arrive at the following.
\begin{lemma}
\label{lem3.2}
Suppose that GRH and the Ramanujan-Petersson conjecture are true. For any odd, positive integer $q$, we have for $x \geq 1$,
\begin{equation*}
  \sum_{p \leq x} a_{\pi}(p)\chi_{8q} (p)\log p
\ll x^{1/2}\log^{2} (qx).
\end{equation*}
\end{lemma}

    We now define
\begin{equation*}
  E(V;\pi, \chi_{8q}, \hat{\phi}) :=\sum_{p \leq V}
    \frac {a_{\pi}(p)\log p}{\sqrt {p}}\chi_{8q} (p)\hat{\phi} \left( \frac {\log p}{M\log X} \right),
\end{equation*}
   where $V \leq X^{B}$ for some bounded real number $B>0$. \newline

   It follows from Lemma \ref{lem3.2} and partial summation that
\begin{equation}
\label{3.1}
 E(V;\pi, \chi_{8q}, \hat{\phi}) \ll \log^{3} (qX).
\end{equation}

   Now on writing $d=l^2m $, we see that
\begin{equation} \label{error1}
\begin{split}
    S_R(X; \pi, \hat{\phi},W) &=
     \sum_{\substack{l>Z \\ (l, 2)=1}}\mu(l)\sum_{(m,2)=1}
     W \left( \frac{l^2m}{X} \right) E(Y;\chi_{8l^2m}, \hat{\phi})
     \ll \sum_{l>Z}\sum_{X/l^2 \leq m \leq
   2X/l^2} \log^{3}(X) \ll \frac {X \log^{3} X}{Z}.
\end{split}
\end{equation}

\subsection{Estimation of $S_M(X; \pi, \hat{\phi},W)$}

   First note that we have, after interchanging sumations,
\begin{align*}
 S_M(X; \pi, \hat{\phi}, W) = \sum_{p}
    \frac {a_{\pi}(p)\log p}{\sqrt {p}} \hat{\phi} \left( \frac {\log p}{M\log X} \right) \left( \frac{2}{p} \right)
\sum_{(d,2)=1}M_Z(d)\left( \frac{d}{p} \right) W\left( \frac{d}{X} \right).
\end{align*}

   Observe that
\begin{align*}
 \sum_{(d,2)=1}M_Z(d)\left( \frac{d}{p} \right) W\left( \frac{d}{X} \right)=\sum_{\substack{\alpha \leq Z \\ (\alpha, 2p)=1}}\mu(\alpha)
 \sum_{(d,2)=1}\left( \frac{d}{p} \right) W\left( \frac{d\alpha^2}{X} \right).
\end{align*}

   We recast the inner sum above as
\begin{align*}
 \sum_{(d,2)=1}\left( \frac{d}{p} \right) W\left( \frac{d\alpha^2}{X} \right)=\sum_{d}\left( \frac{d}{p} \right) W\left(
\frac{d\alpha^2}{X} \right)-\left( \frac{2}{p} \right)\sum_{d}\left( \frac{d}{p} \right) W\left( \frac{2 d\alpha^2}{X} \right).
\end{align*}

Now applying the Poisson summation formula in \eqref{PoissonsumQw} renders the above as
\begin{align*}
 \sum_{(d,2)=1}\left( \frac{d}{p} \right) W\left( \frac{d\alpha^2}{X} \right)=\frac {X}{\alpha^2\sqrt{p}}\sum_{m\geq1}\left( \frac{m}{p} \right)
\widetilde{W}_{\af}\left( \frac{mX}{\alpha^2p} \right)-\frac {X}{2\alpha^2\sqrt{p}}\sum_{m\geq 1}\left( \frac{2m}{p} \right)
\widetilde{W}_{\af}\left( \frac{mX}{2\alpha^2p} \right).
\end{align*}

This enables us to write $S_M(X; \pi, \hat{\phi}, W)$ as
\begin{align*}
 S_M(X; \pi, \hat{\phi}, W) =& X\sum_{(p,2)=1}
    \frac {a_{\pi}(p)\log p}{p} \hat{\phi} \left( \frac {\log p}{M\log X} \right)  \sum_{\substack { \alpha \leq Z \\ (\alpha,
2p)=1}}\frac
{\mu(\alpha)}{\alpha^2} \sum_{m\geq 1}\left( \frac{2m}{p} \right) \widetilde{W}_{\af}\left( \frac{mX}{\alpha^2p} \right) \\
& \hspace*{1cm} -\frac {X}{2} \sum_{(p,2)=1}
    \frac {a_{\pi}(p)\log p}{p} \hat{\phi} \left( \frac {\log p}{M\log X} \right)  \sum_{\substack { \alpha \leq Z \\ (\alpha,
2p)=1}}\frac
{\mu(\alpha)}{\alpha^2} \sum_{m\geq 1}\left( \frac{m}{p} \right) \widetilde{W}_{\af}\left( \frac{mX}{\alpha^2p} \right) \\
:=& S_1-S_2.
\end{align*}

    As the estimations are similar, we only deal with $S_1$ in what follows. We recast it as
\begin{align*}
 S_1 =& X \sum_{\substack { \alpha \leq Z \\ (\alpha, 2)=1}}\frac
{\mu(\alpha)}{\alpha^2}  \sum_{m \geq 1} \sum_{(p,2\alpha)=1}
    \frac {a_{\pi}(p)\log p}{p} \hat{\phi} \left( \frac {\log p}{M\log X} \right) \left( \frac{2m}{p} \right) \widetilde{W}_{\af}\left(
\frac{mX}{\alpha^2p} \right) \\
=& X \sum_{\substack { \alpha \leq Z \\ (\alpha, 2)=1}}\frac
{\mu(\alpha)}{\alpha^2}  \sum_{m \geq 1} \sum_{\substack{(p,\alpha)=1\\ p \equiv 1 \shortmod 4}}
    \frac {a_{\pi}(p)\log p}{p} \hat{\phi} \left( \frac {\log p}{M\log X} \right) \left( \frac{2m}{p} \right) \widetilde{W}_{0}\left(
\frac{mX}{\alpha^2p} \right) \\
& \hspace*{1cm} +X \sum_{\substack { \alpha \leq Z \\ (\alpha, 2)=1}}\frac
{\mu(\alpha)}{\alpha^2}  \sum_{m \geq 1} \sum_{\substack{(p,\alpha)=1\\ p \equiv -1 \shortmod 4}}
    \frac {a_{\pi}(p)\log p}{p} \hat{\phi} \left( \frac {\log p}{M\log X} \right) \left( \frac{2m}{p} \right) \widetilde{W}_{1}\left(
\frac{mX}{\alpha^2p} \right).
\end{align*}

   We then use $\frac 12 ( 1 \pm \chi_{-1}(p))$ to detect the conditions that $p \equiv \pm 1 \pmod 4$, thus further decomposing $S_1$ as
\begin{align*}
 S_1 =& S_{0, 1}+S_{0,-1}+S_{1, 1}-S_{1,-1},
\end{align*}
 where
\begin{align*}
 S_{i,j} =&  X \sum_{\substack { \alpha \leq Z \\ (\alpha, 2)=1}}\frac
{\mu(\alpha)}{2\alpha^2}  \sum_{m \geq 1} \sum_{\substack{(p,2\alpha)=1}}
    \frac {a_{\pi}(p)\log p}{p} \hat{\phi} \left( \frac {\log p}{M\log X} \right) \left( \frac{2jm}{p} \right) \widetilde{W}_{i}\left(
\frac{mX}{\alpha^2p} \right).
\end{align*}

  In the above expression for $S_{i,j}$, we may assume that $Z \leq X$ (in fact, our later choice of $Z$ is much smaller than $X$).  Again, as the estimations are similar, we only consider $S_{0,1}$ in what follows.  The condition $(p, \alpha)=1$ is detected
using $\chi_{\alpha^2}(p)$ and we use \eqref{3.1} to estimate $S_{0,1}$.  Mindful of the size of $Z$, we deduce via partial
   summation that
\begin{equation*}
\begin{split}
\sum_{\substack{(p,2\alpha)=1}} &
    \frac {a_{\pi}(p)\log p}{p} \hat{\phi} \left( \frac {\log p}{M\log X} \right) \left( \frac{2m}{p} \right) \widetilde{W}_{0}\left(
\frac{mX}{\alpha^2p} \right) =\int\limits^{Y}_1\frac {1}{\sqrt{V}}\widetilde{W}_0
    \left( \frac {m X}{\alpha^2V} \right) \dif E(V;\chi_{8\alpha^2m}, \pi, \hat{\phi}) \\
 \ll & \log^{3}(X(|m|+2))
    \left( \frac {1}{\sqrt{Y}}
   \left |\widetilde{W}\left (\frac {m
   X}{\alpha^2Y} \right ) \right |  +  \int\limits^{Y}_{1}\frac
   {1}{V^{3/2}}\left |\widetilde{W}
    \left (\frac {m X}{\alpha^2V} \right ) \right |  \dif V \right.  \left. + \int\limits^{Y}_{1}\frac
   {X}{\alpha^2V^{5/2}} \left |m\widetilde{W}'
    \left(\frac {m X}{\alpha^2V} \right ) \right | \dif V \right).
\end{split}
\end{equation*}

    It follows that
\begin{equation*}
 S_{0,1} \ll X\sum_{\alpha \leq Z}\frac {1}{\alpha^2}(R_1+R_2+R_3),
\end{equation*}
    where
\[ R_1 = \frac {1}{\sqrt{Y}} \sum_{m \neq 0}\log^{3}(X(|m|+2))\left |\widetilde{W}\left (\frac {m
   X}{\alpha^2Y}\right ) \right |, \;  R_2 = \int\limits^{Y}_{1}\frac
   {1}{V^{3/2}}\sum_{m \geq 1}\log^{3}(X(m+2)) \left |\widetilde{W}
    \left (\frac {m X}{\alpha^2V} \right ) \right | \dif V \]
and
\[   R_3 = \int\limits^{Y}_{1}\frac
   {X}{\alpha^2V^{5/2}}\sum_{m \geq 1}\log^{3}(X(m+2)) \left|m\widetilde{W}'
    \left(\frac {m X}{\alpha^2V} \right) \right | \dif V. \]

  Estimations for the above quantities can be found in \cite[(3.4)]{G&Zhao2} by setting $U=1$ there.  Hence we obtain that
\begin{equation}
\label{error4}
  S_{0,1} \ll Z \sqrt{Y}\log^{7} X.
\end{equation}

\subsection{Conclusion }  We now combine the bounds \eqref{error1},
\eqref{error4} and take $Z=\log^3 X$ with any fixed $\varepsilon>0$ to conclude that the estimation given in \eqref{Sorder} holds.  This completes the proof of Theorem \ref{quadraticmainthm}.

\vspace*{.5cm}

\noindent{\bf Acknowledgments.}   P. G. is supported in part by NSFC grant 11871082 and L. Z. by the Faculty Silverstar Grant PS65447 at the University of New South Wales (UNSW).  The authors would also like to thank the anonymous referee of his/her careful inspection of the paper and many helpful comments and suggestions.

\bibliography{biblio}

\begin{bibdiv}
\begin{biblist}

\bib{A&B}{article}{
      author={Andrade, J.~C.},
      author={Baluyot, S.},
       title={Small zeros of {D}irichlet {$L$}-functions of quadratic
  characters of prime modulus},
        date={2020},
     journal={Res. Number Theory},
      volume={6},
      number={2},
       pages={1\ndash 20},
}

\bib{B&F}{article}{
      author={Bui, H.~M.},
      author={Florea, A.},
       title={Zeros of quadratic {D}irichlet {$L$}-functions in the
  hyperelliptic ensemble},
        date={2018},
     journal={Trans. Amer. Math. Soc.},
      volume={370},
      number={11},
       pages={8013\ndash 8045},
}

\bib{BFK21-1}{article}{
      author={Bui, H.~M.},
      author={Florea, A.},
      author={Keating, J.~P.},
       title={Type-{I} contributions to the one and two level densities of
  quadratic {D}irichlet {$L$}-functions over function fields},
        date={2021},
     journal={J. Number Theory},
      volume={221},
       pages={389\ndash 423},
}

\bib{BG}{article}{
      author={Bump, D.},
      author={Ginzburg, D.},
       title={Symmetric square {$L$}-functions on {${\rm GL}(r)$}},
        date={1992},
     journal={Ann. of Math. (2)},
      volume={136},
      number={1},
       pages={137\ndash 205},
}

\bib{Da}{book}{
      author={Davenport, H.},
       title={{M}ultiplicative {N}umber {T}heory},
     edition={Third edition},
      series={Graduate Texts in Mathematics},
   publisher={Springer-Verlag},
     address={Berlin, etc.},
        date={2000},
      volume={74},
}

\bib{DGH}{article}{
      author={Diaconu, A.},
      author={Goldfeld, D.},
      author={Hoffstein, J.},
       title={Multiple {D}irichlet series and moments of zeta and
  {$L$}-functions},
        date={2003},
     journal={Compositio Math.},
      volume={139},
      number={3},
       pages={297\ndash 360},
}

\bib{ER-G&R}{article}{
      author={Entin, A.},
      author={Roditty-Gershon, E.},
      author={Rudnick, Z.},
       title={Low-lying zeros of quadratic dirichlet {$L$}-functions,
  hyper-elliptic curves and random matrix theory},
        date={2013},
     journal={Geom. Funct. Anal.},
      volume={23},
      number={4},
       pages={1230\ndash 1261},
}

\bib{FPS2}{article}{
      author={Fiorilli, D.},
      author={Parks, J.},
      author={S\"{o}dergren, A.},
       title={Low-lying zeros of elliptic curve {$L$}-functions: beyond the
  ratios conjecture},
        date={2016},
     journal={Math. Proc. Cambridge Philos. Soc.},
      volume={160},
      number={2},
       pages={315\ndash 351},
}

\bib{FPS}{article}{
      author={Fiorilli, D.},
      author={Parks, J.},
      author={S\"{o}dergren, A.},
       title={Low-lying zeros of quadratic {D}irichlet {$L$}-functions: lower
  order terms for extended support},
        date={2017},
     journal={Compos. Math.},
      volume={153},
      number={6},
       pages={1196\ndash 1216},
}

\bib{FPS1}{article}{
      author={Fiorilli, D.},
      author={Parks, J.},
      author={S\"{o}dergren, A.},
       title={Low-lying zeros of quadratic {D}irichlet {$L$}-functions: a
  transition in the ratios conjecture},
        date={2018},
     journal={Q. J. Math.},
      volume={69},
      number={4},
       pages={1129\ndash 1149},
}

\bib{Gao}{article}{
      author={Gao, P.},
       title={$n$-level density of the low-lying zeros of quadratic {D}irichlet
  {$L$}-functions},
        date={2014},
     journal={Int. Math. Res. Not.},
      volume={2014},
      number={6},
       pages={1699 \ndash 1728},
}

\bib{Gao2}{article}{
      author={Gao, P.},
       title={The fourth moment of central values of quadratic {H}ecke
  {$L$}-functions in the {G}aussian field},
        date={Preprint},
        note={arXiv:2004.12528},
}

\bib{G&Zhao2}{article}{
      author={Gao, P.},
      author={Zhao, L.},
       title={One level density of low-lying zeros of families of
  {$L$}-functions},
        date={2011},
     journal={Compos. Math.},
      volume={147},
      number={1},
       pages={1\ndash 18},
}

\bib{G&Zhao4}{article}{
      author={Gao, P.},
      author={Zhao, L.},
       title={One level density of low-lying zeros of quadratic and quartic
  {H}ecke {$L$}-functions},
        date={2020},
     journal={Canad. J. Math.},
      volume={72},
      number={2},
       pages={427\ndash 454},
}

\bib{G&Zhao6}{article}{
      author={Gao, P.},
      author={Zhao, L.},
       title={One level density of low-lying zeros of quadratic {H}ecke
  {$L$}-functions to prime moduli},
        date={2020},
     journal={Hardy-Ramanujan J.},
      volume={43},
       pages={173\ndash 187},
}

\bib{G&Zhao9}{article}{
      author={Gao, P.},
      author={Zhao, L.},
       title={One-level density of low-lying zeros of quadratic {H}ecke
  {$L$}-functions of imaginary quadratic number fields},
        date={2022},
     journal={J. Aust. Math. Soc.},
      volume={112},
      number={2},
       pages={170\ndash 192},
}

\bib{G&Zhao5}{article}{
      author={Gao, P.},
      author={Zhao, L.},
       title={Lower-order terms of the one-level density of a family of
  quadratic {H}ecke {$L$}-functions},
        date={to appear},
     journal={J. Aust. Math. Soc.},
}

\bib{G&L}{article}{
      author={Goldmakher, L.},
      author={Louvel, B.},
       title={A quadratic large sieve inequality over number fields},
        date={2013},
     journal={Math. Proc. Cambridge Philos. Soc.},
      volume={154},
      number={2},
       pages={193\ndash 212},
}

\bib{HKS}{article}{
      author={Huynh, D.~K.},
      author={Keating, J.~P.},
      author={Snaith, N.~C.},
       title={Lower order terms for the one-level density of elliptic curve
  {$L$}-functions},
        date={2009},
     journal={J. Number Theory},
      volume={129},
      number={12},
       pages={2883\ndash 2902},
}

\bib{HMM}{article}{
      author={Huynh, D.~K.},
      author={Miller, S.~J.},
      author={Morrison, R.},
       title={An elliptic curve test of the {$L$}-functions ratios conjecture},
        date={2011},
     journal={J. Number Theory},
      volume={131},
      number={6},
       pages={1117\ndash 1147},
}

\bib{iwakow}{book}{
      author={Iwaniec, H.},
      author={Kowalski, E.},
       title={{A}nalytic {N}umber {T}heory},
      series={American Mathematical Society Colloquium Publications},
   publisher={American Mathematical Society},
     address={Providence},
        date={2004},
      volume={53},
}

\bib{ILS}{article}{
      author={Iwaniec, H.},
      author={Luo, W.},
      author={Sarnak, P.},
       title={Low lying zeros of families of ${L}$-functions},
        date={2000},
     journal={Inst. Hautes Etudes Sci. Publ. Math.},
      volume={91},
       pages={55\ndash 131},
}

\bib{KS1}{article}{
      author={Katz, N.},
      author={Sarnak, P.},
       title={Zeros of zeta functions and symmetries},
        date={1999},
     journal={Bull. Amer. Math. Soc.},
      volume={36},
      number={1},
       pages={1\ndash  26},
}

\bib{K&S}{book}{
      author={Katz, N.},
      author={Sarnak, P.},
       title={Random matrices, frobenius eigenvalues, and monodromy},
      series={American Mathematical Society Colloquium Publications},
   publisher={American Mathematical Society},
     address={Providence},
        date={1999},
      volume={45},
}

\bib{Langlands}{incollection}{
      author={Langlands, R.~P.},
       title={Problems in the theory of automorphic forms {\rm in lectures in
  modern analysis and applications, }},
        date={1970},
   booktitle={Lecture notes in math., vol. 170,},
       pages={18\ndash 61.},
}

\bib{Levinson}{article}{
      author={Levinson, N.},
       title={More than one third of zeros of {R}iemann's zeta-function are on
  $\sigma=1/2$},
        date={1974},
     journal={Adv. Math.},
      volume={13},
       pages={383\ndash 436},
}

\bib{MS}{article}{
      author={Mason, A.~M.},
      author={Snaith, N.~C.},
       title={Symplectic $n$-level densities with restricted support},
        date={2016},
     journal={Random Matrices Theory Appl.},
      volume={5},
      number={4},
       pages={36pp.},
}

\bib{MS1}{article}{
      author={Mason, A.~M.},
      author={Snaith, N.~C.},
       title={Orthogonal and symplectic $n$-level densities},
        date={2018},
     journal={Mem. Amer. Math. Soc.},
      volume={251},
      number={1194},
       pages={v+93pp},
}

\bib{M&T-B}{article}{
      author={Milinovich, M.~B.},
      author={Turnage-Butterbaugh, C.~L.},
       title={Moments of products of automorphic {$L$}-functions},
        date={2014},
     journal={J. Number Theory},
      volume={139},
       pages={175\ndash 204},
}

\bib{Miller1}{article}{
      author={Miller, S.~J.},
       title={A symplectic test of the {$L$}-functions ratios conjecture},
        date={2008},
     journal={Int. Math. Res. Not. IMRN},
        note={Art. ID rnm146, 36 pp.},
}

\bib{O&S}{article}{
      author={{\"O}zl{\"u}k, A.~E.},
      author={Snyder, C.},
       title={On the distribution of the nontrivial zeros of quadratic
  {$L$}-functions close to the real axis},
        date={1999},
     journal={Acta Arith.},
      volume={91},
      number={3},
       pages={209\ndash 228},
}

\bib{Ru}{article}{
      author={Rubinstein, M.~O.},
       title={Low-lying zeros of {$L$}-functions and random matrix theory},
        date={2001},
     journal={Duke Math. J.},
      volume={209},
      number={1},
       pages={147\ndash  181},
}

\bib{R&S}{article}{
      author={Rudnick, Z.},
      author={Sarnak, P.},
       title={Zeros of principal {$L$}-functions and random matrix theory},
        date={1996},
     journal={Duke Math. J.},
      volume={81},
      number={2},
       pages={269\ndash 322},
}

\bib{sound1}{article}{
      author={Soundararajan, K.},
       title={Nonvanishing of quadratic {D}irichlet {$L$}-functions at
  $s=\frac{1}{2}$},
        date={2000},
     journal={Ann. of Math. (2)},
      volume={152},
      number={2},
       pages={447\ndash 488},
}

\end{biblist}
\end{bibdiv}
\bibliographystyle{amsxport}

\vspace*{.5cm}

\noindent\begin{tabular}{p{6cm}p{6cm}p{6cm}}
School of Mathematical Sciences & School of Mathematics and Statistics \\
Beihang University & University of New South Wales \\
Beijing 100191 China & Sydney NSW 2052 Australia \\
Email: {\tt penggao@buaa.edu.cn} & Email: {\tt l.zhao@unsw.edu.au} \\
\end{tabular}

\end{document}